\documentclass{compositio}
\usepackage[utf8]{inputenc}
\usepackage[english]{babel}
\usepackage{newlfont}
\usepackage{amsmath,amssymb,amsfonts,amsbsy,bbm}
\usepackage{mathtools,braket,mathrsfs}
\usepackage[all]{xy}
\usepackage{tikz}
\usetikzlibrary{arrows}
\usepackage{enumerate}
\usepackage{stmaryrd}
\usepackage{hyperref}
\usepackage{upgreek}
\usepackage[new]{old-arrows}
\usepackage{extarrows}
\usepackage{textcomp}

\newcommand{\mP}{\mathrm{P}}
\newcommand{\mQ}{\mathrm{Q}}
\newcommand*{\bfcdot}{\scalebox{0.75}{$\bullet$}}
\newcommand{\Z}{\ensuremath{\mathbb{Z}}}
\newcommand{\Q}{\ensuremath{\mathbb{Q}}}

\newcommand{\C}{\ensuremath{\mathbb{C}}}

\newcommand{\p}{\ensuremath{\mathfrak{p}}}
\newcommand{\q}{\ensuremath{\mathfrak{q}}}
\newcommand{\n}{\ensuremath{\mathfrak{n}}}

\newcommand{\too}{\longrightarrow}								
\newcommand{\mapstoo}{\longmapsto}
\newcommand{\ontoo}{\relbar\joinrel\twoheadrightarrow}

\newcommand{\into}{\hookrightarrow}
\newcommand{\onto}{\twoheadrightarrow}

\newcommand{\cf}{\mathbbm{1}}
\newcommand{\dd}{\textnormal{d}}

\DeclareMathOperator{\Hom}{Hom}

\DeclareMathOperator{\Sym}{Sym}
\DeclareMathOperator{\Ten}{T}

\DeclareMathOperator{\Gal}{Gal}

\DeclareMathOperator{\id}{id}

\DeclareMathOperator{\No}{N}

\DeclareMathOperator{\Ind}{Ind}

\DeclareMathOperator{\GL}{GL}

\DeclareMathOperator{\rec}{rec}
\DeclareMathOperator{\pr}{pr}

\newcommand{\an}{\mathrm{an}}
\newcommand{\alg}{\mathrm{alg}}
\newcommand{\Tate}{\mathrm{Tate}}

\DeclareMathOperator{\Ad}{Ad}
\DeclareMathOperator{\add}{add}
\DeclareMathOperator{\nm}{nm}

\DeclareMathOperator{\res}{\operatorname{res}}

\renewcommand{\det}{\operatorname{det}}
\DeclareMathOperator{\ord}{ord}

\def\XXint#1#2#3{{\setbox0=\hbox{$#1{#2#3}{\int}$ }
\vcenter{\hbox{$#2#3$ }}\kern-.585\wd0}}

\makeatletter
\newcommand{\exterior}[1]{\mathop{\mathpalette\exterior@{#1}}}
\newcommand{\exterior@}[2]{%
  \raisebox{\depth}{%
  \fontsize{\sf@size}{0}%
  \m@th
  $\ifx#1\displaystyle\textstyle\else#1\fi\bigwedge$}%
  ^{\mspace{-2mu}#2}%
  \kern-\scriptspace
}
\makeatother

\theoremstyle{plain}
\newtheorem{theorem}{Theorem}[section]
\newtheorem{lemma}[theorem]{Lemma}
\newtheorem{proposition}[theorem]{Proposition}
\newtheorem{corollary}[theorem]{Corollary}
\newtheorem{conjecture}[theorem]{Conjecture}

\newtheorem{thmx}{Theorem}

\theoremstyle{definition}
\newtheorem{remark}[theorem]{Remark}

\newtheorem{assumptions}[theorem]{Assumption}
\newtheorem{definition}[theorem]{Definition}

\begin{document}

\title{On the algebraicity of polyquadratic plectic points}

\author{Michele Fornea}
\email{mfornea@math.columbia.edu}
\address{Columbia University, New York, USA.}

\author{Lennart Gehrmann}
\email{lennart.gehrmann@uni-due.de}
\address{Universität Duisburg-Essen, Essen, Germany.}


\begin{abstract}
We establish direct evidence of the arithmetic significance of \emph{plectic Stark--Heegner points} for elliptic curves of arbitrarily large rank.
The main contribution is a proof of the algebraicity of plectic points associated to polyquadratic CM extensions of totally real number fields.
Moreover, we relate the non-vanishing of plectic points to analytic and algebraic ranks of elliptic curves. 
\end{abstract}

\maketitle

\tableofcontents

%%%%%%%%%%%%%%%%%%%%%%%%%%%%%%%%%%%%%%%%%%%%%%%%%%%%%%
% Introduction
%%%%%%%%%%%%%%%%%%%%%%%%%%%%%%%%%%%%%%%%%%%%%%%%%%%%%%

\section{Introduction}
In previous work, \emph{plectic Stark--Heegner points} were associated to quadratic extensions $E/F$ of number fields and modular elliptic curves $A_{/F}$ under some technical assumptions.
The construction generalizes the description of classical Heegner points on Shimura curves admitting a $p$-adic uniformization, and combines Nekov\'a$\check{\text{r}}$--Scholl's plectic insights (\cite{PlecticNS}, \cite{NekRubinfest}) with Darmon's pioneering work on Stark-Heegner points \cite{IntegrationDarmon}.
Moreover, Conjectures 1.3, 1.5 of \cite{plecticHeegner} predict the algebraicity of plectic Stark--Heegner points and their significance for elliptic curves of higher rank, for which
some numerical evidence was provided in \cite{PlecticInvariants}.
The aim of this paper is to establish direct evidence for the aforementioned conjectures in the polyquadratic CM case.

\subsection{Conjectures on plectic Stark--Heegner points}\label{conjectures}
Even though the formulation of the conjectures does does not require any restriction on the possible signatures of the fields $E$ and $F$, we only consider CM extensions in this article; that is, $F$ is a totally real number field and $E/F$ a is totally complex quadratic extension.
All prime divisors of the conductor $\mathfrak{f}$ of the elliptic curve $A$ are assumed to be unramified in $E/F$.
We fix a rational prime $p$ and a set $S=\{\p_1,\ldots,\p_r\}$ of $p$-adic primes of $F$ such that
\begin{itemize}
\item $A_{/F}$ has multiplicative reduction at every $\p \in S$,
\item the primes in $S$ are all inert in $E$. 
\end{itemize}
Given $\p\in S$ we also denote by $\p$ the unique prime of $E$ above $\p$ .
Furthermore, we let $a_\p\in\{\pm1\}$ equal $+1$ (resp.~$-1$) when $A$ has split (resp.~non-split) multiplicative reduction at $\p$.
The quantity $a_\p$ is closely related to the local root number $\varepsilon_{\p}(A/F)$ of $A$:
\[
a_\p=-\varepsilon_{\p}(A/F).
\]
By setting $p_S:=\p_1\cdots\p_r$, we may write
\[
\mathfrak{f}=p_S\cdot \n^{\sharp}\cdot \n^{\flat},
\]
for coprime ideals $\n^{\sharp}, \n^{\flat}\in \mathcal{O}_F$ such that
$\n^{\sharp}$ is divisible by \emph{every} prime divisor of $\mathfrak{f}$ split in $E/F$.
\begin{assumptions}[(Plectic Heegner hypothesis for $(A,E,S)$)]\label{AssHeegner2}
  We require that
	\begin{itemize}
		\item $\n^{\flat}$ is square-free,
		\item the number of prime factors of $\n^{\flat}$ is congruent to $[F:\Q]$ modulo $2$. 
	\end{itemize}
\end{assumptions}
\noindent Since $E/F$ is a CM extension, Assumption \ref{AssHeegner2} implies that the sign of the functional equation of the $L$-function $L(A/E,s)$ of the base-change of $A_{/F}$ to $E$ is equal to $\varepsilon(A/E)=(-1)^{r}$.

\vspace{1.5mm}
\noindent  We introduce the following useful notations.
First, if $H$ is any commutative prodiscrete group, we denote by $\widehat{H}$ the torsion-free part of its pro-$p$ completion.
Second, if $M$ is an abelian group and $\Omega$ a field of characteristic zero, we use $M_\Omega$ and $\exterior{r}(M_\Omega)$ to respectively denote the tensor product $M\otimes_\Z \Omega$, and the $r$-th exterior power of the $\Omega$-vector space $M_\Omega$.

\vspace{1.5mm}
\noindent Consider the tensor product
\[
\widehat{A}(E_{S}):=\widehat{A}(E_{\p_1}) \otimes_{\Z_p} \ldots \otimes_{\Z_p} \widehat{A}(E_{\p_r}).
\]
Under the running assumptions, the plectic Stark--Heegner point (for the trivial character) is
\[
\mP_{A,S}\in\widehat{A}(E_{S})_\Q
\]
as constructed in (\cite{plecticHeegner}, Section 4.4), where it was denoted $\mP^{\cf}_{A}$.
The notion of algebraicity for plectic Stark--Heegner points is formulated in terms of a determinant map.
Writing $\iota_\mathfrak{p}\colon E\hookrightarrow E_\mathfrak{p}$ for the canonical embedding for every $\mathfrak{p}\in S$, we can consider the homomorphism
\[
\det \colon\exterior{r}(A(E)_\Q)\too \widehat{A}(E_{S})_\Q\\
\]
 given by
\[
\det\big(P_1\wedge\dots\wedge P_r\big)=\det \begin{pmatrix}
		\iota_{\mathfrak{p}_1}(P_1)&\dots& \iota_{\mathfrak{p}_r}(P_1)\\
	\vdots&\ddots&\vdots\\
	\iota_{\mathfrak{p}_{1}}(P_r)&\dots& \iota_{\mathfrak{p}_r}(P_r)
	\end{pmatrix}.
\]
As usual, $r_\alg(A/E)$ denotes the rank of the finitely generated abelian group $A(E)$, and $r_\an(A/E)$ the order of vanishing of the $L$-function $L(A/E,s)$ at $s=1$.
\begin{conjecture}\label{firstconjecture}
If $r_\alg(A/E)\geq r$, then there exists an element $w_{A,S}\in \exterior{r} (A(E)_\Q)$ such that 
	\[
	\mP_{A,S}=\det(w_{A,S}).
	\]	
Moreover, if $\mP_{A,S} \neq 0$, then $r_\alg(A/E)=r$.
\end{conjecture}

\begin{remark}
When $S$ consists of a single prime, Conjecture \ref{firstconjecture} follows from the generalization of the Gross--Zagier--Kolyvagin theorem for totally real number fields by Nekov\'a$\check{\text{r}}$ and Zhang. Indeed, \v{C}erednik--Drinfeld's uniformization of Shimura curves implies that the plectic point $\mP_{A,S}$ is a Heegner point when $\lvert S\rvert=1$ (see Section \ref{Heegner} for more details).
\end{remark}

\noindent There are also plectic Stark--Heegner points associated to non-trivial anticyclotomic characters of $E/F$. Aiming for clarity in the introduction, we discuss them only in the body of the paper.

\subsubsection{Eigenspaces for partial Frobenii.}
Let $\sigma_\p$ be the generator of the Galois group of $E_\p/F_\p$.
It naturally acts on $\widehat{A}(E_\p)$, and thus also on $\widehat{A}(E_S)$ via its action on the $\p$-th factor.
We set
\[
\widehat{A}(E_\p)^{\pm}:=\widehat{A}(E_\p)^{\sigma_\p =\pm a_\p}_\Q\quad \mbox{and}\quad \widehat{A}(E_S)^{\pm}:= \otimes_{\p\in S} \widehat{A}(E_\p)^{\pm}. 
\]
There are two eigenspace projections
\[
\pr^{\pm}_S\colon\widehat{A}(E_S)_\Q\rightarrow \widehat{A}(E_S)^{\pm},\qquad \pr^{\pm}_S=\prod_{\p\in S}(1\pm a_{\p}\cdot\sigma_{\mathfrak{p}}).
\]
The main results of this article (Theorems \ref{mainthmGZK} \& \ref{mainthmALG}) establish the first cases of the minus part of the following conjecture, a direct consequence of Conjecture \ref{firstconjecture}.
\begin{conjecture}\label{plusminusconjecture}
If $r_\alg(A/E)\geq r$, then
there exists an element $w_{A,S}\in \exterior{r}(A(E)_\Q)$ such that 
	\[
	\pr^{\pm}_S(\mP_{A,S})=\pr^{\pm}_S(\det(w_{A,S})).
	\]	
Moreover, if $\pr^{\pm}_S(\mP_{A,S}) \neq 0$, then $r_\alg(A/E)= r$.
\end{conjecture}
\noindent The special cases that we treat in Theorems \ref{mainthmGZK}, \ref{mainthmALG} and \ref{mainthm3} are singled out precisely to leverage the known properties of classical Heegner points.
The key idea is to further suppose that $F$ is a polyquadratic extension of another totally real number field $F_\circ$, the elliptic curve is the base change of an elliptic curve $A_\circ$ defined over $F_\circ$, and $E$ is the compositum of $F$ with a quadratic CM extension $E_\circ/F_\circ$.
Then, under an appropriate Heegner hypothesis, we use Heegner points for $A_{\circ/E_\circ}$ and its twists by characters of $\Gal(F/F_\circ)$, plus a factorization of anticyclotomic $p$-adic $L$-functions to establish our results.

\subsubsection{Plectic $p$-adic invariants.}
Anticyclotomic $p$-adic $L$-functions come in to play in the proofs of our theorems because of the $p$-adic Gross-Zagier formula  (\cite{plecticHeegner}, Theorem A) relating higher order derivatives to plectic $p$-adic invariants. These invariants, denoted  $\mQ_{A,S}$, are canonical lifts  of the points $\pr^{-}_S(\mP_{A,S})$ with respect to a ``plectic'' Tate parametrization:
as the elliptic curve $A_{/F}$ has multiplicative reduction at every $\p \in S$, Tate's $p$-adic uniformization results provides surjections $\phi^{\Tate}_\p\colon E_\p^{\times}\twoheadrightarrow A(E_\p)$
whose kernels are generated by \emph{Tate periods} $q_\p\in F_\p^\times\setminus \mathcal{O}_{F_\p}^\times$.
If we denote by $\widehat{E}_\p^{-}:=(\widehat{E}_\p^\times)^{\sigma_\p=-1}$ the subgroup of $\widehat{E}_\p^\times$ on which $\sigma_\p$ acts via inversion, and we set
\[
\widehat{E}_{S,\otimes}^{-}:= \widehat{E}_{\p_1}^{-}\otimes_{\Z_p}\ldots \otimes_{\Z_p}\widehat{E}_{\p_r}^{-},
\]
then the plectic $p$-adic invariant $\mQ_{A,S}$ is the unique element of $\widehat{E}_{S,\otimes}^{-}$ satisfying
\[
\phi^{\Tate}_S(\mQ_{A,S})= \pr^{-}_S(\mP_{A,S}).
\]
Here $\phi^{\Tate}_S\colon \widehat{E}_{S,\otimes}^\times\to \widehat{A}(E_S)$ denotes the tensor product of Tate's local uniformizations.
Since the restriction of $\phi^{\Tate}_S$ to $\widehat{E}_{S,\otimes}^{-}$ is injective, we have that $\mQ_{A,S}\neq 0$ if and only if $\pr^{-}_S(\mP_{A,S})\neq 0$.

\subsection{The polyquadratic setup}\label{polyquadratic}
For the rest of the introduction we suppose that the totally real number field $F$ is a polyquadratic extension of degree $2^t$ of a number field $F_\circ$, i.e., $F/F_\circ$ is a Galois extension with Galois group $\mathfrak{G}:=\Gal(F/F_\circ)\cong (\Z/2\Z)^t.$
Further, we assume that $E$ is the compositum of $F$ with a quadratic CM extension $E_\circ/F_\circ$ and that the following technical assumptions are satisfied.
\begin{assumptions}\label{AssFields}
	We require that
	\begin{itemize}
		\item every non-trivial subextension of $F/F_\circ$ is ramified,
		\item all primes of $F_\circ$ that ramify in $F$ split in $E_\circ$,
		\item the elliptic curve $A_{/F}$ is the base change of a modular elliptic curve $A_{\circ}$ defined over $F_\circ$, whose conductor $\mathfrak{f}_\circ$ is unramified in $F/F_\circ$,
		\item the set $S$ consists of all primes of $F$ lying above a single prime $\wp$ of $F_\circ$, totally split in $F$.
	\end{itemize}
\end{assumptions}
\begin{remark}
The need for the first assumption is explained in Remark \ref{ExplanationAss}. The splitting in $E_\circ$ of the primes ramified in $F/F_\circ$ is necessary to construct Heegner points on twists of $A_{\circ/F}$ by characters of $\mathfrak{G}$, while the total splitting of the prime $\wp$ in $F/F_\circ$ is just a  simplifying hypothesis for the proof of Proposition \ref{factorizationpLf}.
\end{remark}

\noindent Assumption \ref{AssFields} implies that the elliptic curve $A_{/F}$ is modular by quadratic base change for Hilbert modular forms. Moreover, we deduce that the cardinality $r=\lvert S\rvert$ equals $2^t$, and that the prime $\wp$ is inert in $E_\circ/F_\circ$.
By a small abuse of notation, we denote by $\wp$ the unique prime of $E_\circ$ lying above $\wp$. 
Since $\wp$ is completely split in $F/F_\circ$, the elliptic curve $A_{\circ/F_\circ}$ has multiplicative reduction at $\wp$.
Furthermore, if we set $a_\wp=1$ (resp.~$a_\wp=-1$) in case $A_{\circ/F_\circ}$ has split (resp.~non-split) multiplicative reduction, we have
\begin{equation}\label{typeequality}
a_\p=a_\wp\quad \forall\ \p\in S.
\end{equation}
 Now, write the conductor $\mathfrak{f}_\circ$ of $A_{\circ/F_\circ}$ as
\[
\mathfrak{f}_\circ=\wp\cdot \n_\circ^{\sharp}\cdot \n_\circ^{\flat}
\]
where $\n_\circ^{\sharp}$ is divisible by \emph{every} prime divisor of $\mathfrak{f}_\circ$ split in $E_\circ/F_\circ$.
\begin{assumptions}[(Generalized Heegner hypothesis for $(A_\circ,E_\circ,\wp)$)]\label{AssHeegner}
	We require that
	\begin{itemize}
		\item $\n_\circ^{\flat}$ is square-free,
		\item the number of prime factors of $\n_\circ^{\flat}$ is congruent to $[F_\circ:\Q]$ modulo $2$. 
	\end{itemize}
\end{assumptions}
\noindent Under Assumption \ref{AssHeegner}, the sign of the functional equation for $A_{\circ/E_\circ}$ equals $\varepsilon(A_\circ/E_\circ)=-1$. Hence, the BSD--conjecture predicts that $A_\circ(E_\circ)$ is non-torsion.
Moreover, for any character $\eta\colon \mathfrak{G}\to \{\pm 1\}$, the twist $A^{\eta}_{\circ/F_\circ}$ also fulfils the generalized Heegner hypothesis and we have
\[
\varepsilon(A_\circ^{\eta}/E_\circ)=-1
\]
because every prime ramified in $F/F_\circ$ splits in $E_\circ/F_\circ$  by Assumption \ref{AssFields}.
Thus, we expect that
\[
r_{\alg}(A/E)\stackrel{?}{\geq} [F:F_\circ] =r,
\]
with equality if and only if $r_{\alg}(A^{\eta}_\circ/E_\circ)=1$ for every character $\eta\colon \mathfrak{G}\to \{\pm 1\}$.

\begin{remark}
Under the running assumptions \ref{AssFields} and \ref{AssHeegner}, the plectic Heegner hypothesis for $(A,E,S)$ holds if and only the number of prime divisors of $\n^\flat$ is even.
For example, this is the case when the conductor $\mathfrak{f}_\circ$ of $A_{\circ/F_\circ}$ is completely split in $F/F_\circ$.
\end{remark}

\begin{remark}
Simple examples satisfying all our hypotheses can be found by considering $F_\circ=\Q$, $E_\circ/\Q$ an imaginary quadratic field, $F/\Q$ a real quadratic field, and $A_{\circ}$ a rational elliptic curve of conductor $\mathfrak{f}_\circ=p\cdot q$ for two rational primes both inert in $E_\circ/\Q$, and with $p$ split in $F/\Q$.
\end{remark}

\subsection{Main results}\label{mainresults}
As $\wp$ is completely split in $F$,  we have for every $\p\in S$ canonical identifications $F_{\circ,\wp}= F_\p$ and $E_{\circ,\wp} = E_\p$.
The resulting identifications $\widehat{A_\circ}(E_{\circ,\wp})=\widehat{A}(E_\p)$ are used to define the \emph{norm map}
\[
\No_{S/\wp}\colon \widehat{A}(E_S)
\xlongrightarrow{\sim} \widehat{A_\circ}(E_{\circ,\wp})^{\otimes r}
\too \Sym^{r}_{\Z_p} \big(\widehat{A_\circ}(E_{\circ,\wp})\big),
\]
where the second arrow is the canonical projection.
\begin{remark}
If $\wp$ is of degree one, the restriction of the norm map $\No_{S/\wp}$ to $\widehat{A}(E_S)^{\pm}$ is injective.
\end{remark}

\noindent Under our running assumptions (\ref{AssHeegner2}, \ref{AssFields}, \ref{AssHeegner}), we deduce the following theorems about plectic points from the known properties of Heegner points, the $p$-adic uniformization of Shimura curves, and a factorization of anticyclotomic $p$-adic $L$-functions (Corollary \ref{factorpadicL}).

\begin{thmx}[(Arithmetic significance)]\label{mainthmGZK}
The following implication holds:
 \[  
 \No_{S/\wp}\big(\pr^{-}_S(\mP_{A,S})\big)\neq 0 \quad\implies\quad r_{\alg}(A/E)=r\quad\& \quad r_{\an}(A/E)=r.
 \]
\end{thmx}

\begin{thmx}[(Algebraicity)]\label{mainthmALG}
There is a quadratic extension $\Omega/\Q$ and $w_{A,S}\in \exterior{r} (A(E)_\Omega)$ s.t.
	\[ 
	\No_{S/\wp}\big(\pr^{-}_S(\mP_{A,S})\big)=\No_{S/\wp}\big(\pr^{-}_S(\det(w_{A,S}))\big).
	\]
\end{thmx}

\begin{remark}
Aside from the quadratic extension $\Omega/\Q$, Theorems \ref{mainthmGZK} and \ref{mainthmALG} provide a proof of the minus part of Conjecture \ref{plusminusconjecture} in the polyquadratic setup when $\wp$ is of degree one.
Using the main theorem of \cite{Santi2022} one can apply the same strategy to prove the plus part of the conjecture.
This will be explained in more detail in future work.
\end{remark}

\begin{remark}
The quadratic extension $\Omega/\Q$ is generated by the square-root of a rational number that is the product of various explicit terms: Petersson norms, discriminants, Euler factors and special values of Dedekind zeta functions.
It would be interesting to know whether that rational number is in fact a square.
Similar questions were raised and shown to be implied by the Birch--Swinnerton-Dyer conjecture in \cite{MokSpecialValues}.
\end{remark}

\noindent Now, set $A^+_{/F}=A_{/F}$ and denote by $A^-_{/F}$ the quadratic twist of $A_{/F}$ with respect to the extension $E/F$.
We partition $S=S^+\cup S^-$ by declaring that the subset $S^+\subseteq S$ contains \emph{all} the primes in $S$ of split multiplicative reduction for $A^+_{/F}$, and define $\varrho_{A}(S):=\max\big\{r_{\alg}(A^\pm/F)+\lvert S^\pm\rvert\big\}.$ 
 \begin{thmx}\label{mainthm3}
 We also have
 \[  
 \No_{S/\wp}\big(\pr^{-}_S(\mP_{A,S})\big)\neq 0\quad\iff\quad r_{\an}(A/E)=r\quad\&\quad\varrho_A(S)=r.
 \]
\end{thmx}
\begin{remark}
	If $\wp$ is a prime of degree one, Theorem \ref{mainthm3} establishes (\cite{PlecticInvariants}, Conjecture 1.5) in the polyquadratic CM case.
\end{remark}

\noindent We note that in the main body of this article, we prove generalizations of Theorems  \ref{mainthmGZK}, \ref{mainthmALG} for plectic Stark--Heegner points associated to anticyclotomic characters of $E/F$ that are restrictions of anticyclotomic characters of $E_\circ/F_\circ$.

\begin{acknowledgements}
	We warmly thank Jan Vonk for the numerous conversations related to our work on plectic Stark--Heegner points.
While working on this article, the first named author was a Simons Junior Fellow.
\end{acknowledgements}

%%%%%%%%%%%%%%%%%%%%%%%%%%%%%%%%%%%%%%%%%%%%%%%%%%%%%%
% Preliminaries
%%%%%%%%%%%%%%%%%%%%%%%%%%%%%%%%%%%%%%%%%%%%%%%%%%%%%%

\section{Preliminaries}
We gather some basic results on symmetric powers and completed group algebras.

\subsection{Symmetric powers}
Let us fix a commutative ring $R$.
Given an $R$-module $M$ and an integer $n\geq 0$ we write
\[
M^{\otimes n}:=\underbrace{M\otimes_R\ldots\otimes_R M}_{n\ \text{times}}.
\]
Recall that the \emph{symmetric algebra} $\Sym_R(M)$ of $M$ is the quotient of the tensor algebra
\[
\Ten_R(M):=\bigoplus_{n\geq 0} M^{\otimes n}
\]
by the ideal generated by $x\otimes y -y \otimes x$ for $x,y\in M$.
As this ideal is graded, the natural grading on $\Ten_R(M)$ induces a grading on $\Sym_R(M)$:
\[
\Sym_R(M)=\bigoplus_{n\geq 0} \Sym^{n}_R(M).
\]
We denote the image of an element $m\in M^{\otimes n}$ in $\Sym_R^{n}(M)$ by $[m]$.
Given a homomorphism $f\colon M \to N$ of $R$-modules we write
\[
\Sym_R^{n}(f)\colon \Sym_R^{n}(M)\too \Sym_R^{n}(N)
\]
for the induced homomorphism.
If $M$ and $N$ are $R$-modules, there is a canonical isomorphism
\begin{align}\label{symisom1}
\Sym_R(N\oplus M)=\Sym_R(M)\otimes_R \Sym_R(N)
\end{align}
of $R$-algebras.
Now, suppose $M$ is a finitely generated free $R$-module with generators $m_1,\ldots,m_\ell$, then there is an isomorphism of graded $R$-algebras
\begin{align}\label{symisom2}
R[x_1,\ldots,x_\ell]\xlongrightarrow{\sim}\Sym_R(M),\qquad x_i\mapsto m_i.
\end{align}

\begin{lemma}\label{squaremap}
	Let $R$ be an integrally closed domain, $M$ a finitely generated free $R$-module, and
  \[
	(-)^2\colon\Sym^n_R(M)\too\Sym^{2n}_R(M)
	\]
	the squaring map.
	If $x,y$ are elements of $\Sym^n_R(M)$ and $C\in R\setminus\{0\}$ is a non-zero constant satisfying $x^2=C\cdot y^2$,
	then there exists a square-root $\sqrt{C}\in R$ such that $x=\sqrt{C} \cdot y$.
\end{lemma}
\begin{proof}
As $R$ is an integrally closed domain, equation \eqref{symisom2} implies that the $R$-algebra $\Sym_R(M)$ is one as well.
Thus, the equality $C=(x/y)^2$ in the fraction field of $\Sym_R(M)$ implies that $\sqrt{C}:=x/y$ is an element of $\Sym_R(M)$.
Moreover, $\sqrt{C}\in R$ because its square belongs to $R$.
\end{proof}

\noindent The following lemma can be easily deduced from \eqref{symisom1} and \eqref{symisom2}.
\begin{lemma}\label{symmaps}
Let $M_1,\ldots, M_n$ and $M$ be finitely generated free $R$-modules.
\begin{enumerate}[(a)]
\item\label{symmapsa} The canonical map
\[
\mu\colon M_1\otimes_R\ldots \otimes_R M_n \to \Sym^{n}_R (M_1\oplus\ldots \oplus M_n),
\qquad m_1\otimes\ldots\otimes m_n \mapsto [m_1\otimes\ldots\otimes m_n]
\]
is injective.
\item\label{symmapsb} The following diagram is commutative
\[\xymatrix{
M^{\otimes n}\ar[rr]^\mu\ar[drr]_{[-]} && \Sym_R^{n}(M^{\oplus n})\ar[d]^{\Sym_R^{n}(\id_M^{\oplus n})}\\
&& \Sym_R^{n}(M).
}\]
\end{enumerate}
\end{lemma}

\subsection{Completed group algebras}\label{groupalgebras}
Let $G=\varprojlim_i G_i$ be a topologically finitely generated commutative profinite group.
Recall that the \emph{completed group algebra} of $G$ with coefficients in a commutative ring $R$ is defined as
\[
R\llbracket G\rrbracket=\varprojlim\hspace{0.5mm} R[G_i].
\]
We denote by $(-)^\vee\colon R\llbracket G\rrbracket\to R\llbracket G\rrbracket$ the involution induced by inversion on $G$.
In the rest of this section we always consider the coefficient ring $R=\Z_p$.
Let $I(G)\subseteq \Z_p\llbracket G\rrbracket$ be the \emph{augmentation ideal}, i.e. the kernel of the natural map $\Z_p\llbracket G\rrbracket \onto \Z_p$.
More generally, if $Q$ is a quotient of $G$ by an open subgroup, the \emph{relative augmentation ideal} is defined as
\[
I_Q(G):=\ker(\Z_p\llbracket G\rrbracket \ontoo \Z_p[Q]).
\]
Note that the quotients $I_Q(G)^n/I_Q(G)^{n+1}$ are modules over the group ring $\Z_p[Q]$.
If $\Theta$ is an element of $I_Q(G)^{n}$, we write $\partial^{n}_Q(\Theta)$ for its image in $I_Q(G)^n/I_Q(G)^{n+1}$.
The map $G \to I(G)$, $g \mapsto [g] - 1$
induces an isomorphism of $\Z_p$-modules
\begin{align}\label{augmentationisom}
G\otimes_{\widehat{\Z}}\Z_p \xlongrightarrow{\sim} I(G)/I(G)^2,
\end{align}
and for every integer $n\geq 1$ a surjection of $\Z_p$-modules 
\[
\Sym_{\Z_p}^{n}(G\otimes_{\widehat{\Z}}\Z_p)\ontoo I(G)^{n}/I(G)^{n+1}.
\]
When $G$ is a finitely generated free $\Z_p$-module,
a choice of topological generators $\{g_1,\ldots, g_s\}$ determines an isomorphism
\begin{align}\label{iwasawaisom}
\Z_p\llbracket G\rrbracket \xlongrightarrow{\sim} \Z_p\llbracket t_1,\ldots,t_s\rrbracket,\quad [g_i]\mapstoo t_i+1,
\end{align}
 mapping the augmentation ideal to the ideal $(t_1,\ldots,t_s)$.
It follows that the surjective maps
\begin{align}\label{augmentisom}
\Sym_{\Z_p}^{n}(G)\overset{\sim}{\too} I(G)^{n}/I(G)^{n+1} 
\end{align}
are isomorphisms for all $n\geq 1$.
Furthermore, when $G$ is a product $G=H\times Q$ with $Q$ finite, it is easy to see that the canonical map
\[
I(H)^n/I(H)^{n+1} \otimes_{\Z_p}\Q_p \too I(G)^n/I(G)^{n+1} \otimes_{\Z_p} \Q_p
\]
is an isomorphism for all $n\geq 1$.
The following lemma gives a slight generalization of this fact.
\begin{lemma}\label{injective}
Let $G$ be a finitely generated commutative profinite group, $H\le G$ an open subgroup that is a finitely generated $\Z_p$-module, and $Q$ a finite quotient of $G$ such that $H \subseteq \ker(G\onto Q)$.
Then, the canonical $\Q_p[Q]$-linear map
\[
I(H)^{n}/I(H)^{n+1}\otimes_{\Z_p} \Q_p[Q] \too I_Q(G)^{n}/I_Q(G)^{n+1}\otimes_{\Z_p}\Q_p
\]
is injective for all $n\geq 1$.
\end{lemma}
%Note that by definition there is a natural identification
%	\[
%	I(H)^{n}/I(H)^{n+1}\otimes_{\Z} \Q[Q]= I(H)^{n}/I(H)^{n+1}\otimes_{\Z_p} \Q_p[Q].
%	\]
\noindent Now, both the augmentation ideal $I(G)$ and the relative versions $I_Q(G)$ are clearly stable under $(-)^\vee$.
Equation \eqref{augmentationisom} implies that $(-)^\vee$ induces multiplication with $-1$ on $I(G)/I(G)^2$ and, thus, it induces multiplication with $(-1)^{n}$ on $I(G)^{n}/I(G)^{n+1}$.
This observation readily implies the following relative statement:
under the assumptions of Lemma \ref{injective} the following diagram commutes
\begin{equation}\label{diagramSection2.2}
	\xymatrix{
	I(H)^{n}/I(H)^{n+1}\otimes_{\Z_p} \Z_p[Q]\ar[r]\ar[d]_{(-1)^n\id \otimes (-)^\vee} & I_Q(G)^{n}/I_Q(G)^{n+1}\ar[d]^{(-)^\vee}\\
	I(H)^{n}/I(H)^{n+1}\otimes_{\Z_p} \Z_p[Q]\ar[r] &  I_Q(G)^{n}/I_Q(G)^{n+1}.
	}
\end{equation}

%%%%%%%%%%%%%%%%%%%%%%%%%%%%%%%%%%%%%%%%%%%%%%%%%%%%%%
% Plectic points and p-adic L-functions
%%%%%%%%%%%%%%%%%%%%%%%%%%%%%%%%%%%%%%%%%%%%%%%%%%%%%%

\section{Plectic points and $p$-adic $L$-functions}
We begin by explaining how the construction of plectic Stark-Heegner points recovers the $p$-adic uniformization of classical Heegner points as a special case. Then, we recall the $p$-adic Gross--Zagier formula (\cite{plecticHeegner}, Theorem A) relating plectic points to derivatives of certain anticyclotomic $p$-adic $L$-functions, for which we also state precise interpolation formulas.
Note that in this section we work in the setup of Subsection \ref{conjectures}, that is, $E/F$ is an arbitrary quadratic CM extension and Assumption \ref{AssHeegner2} is supposed to hold.
In particular, we never assume that we are in a polyquadratic situation.

\subsection{Plectic Stark--Heegner points}

\noindent We fix an $\mathcal{O}_F$-ideal $\mathfrak{c}$ coprime to the conductor of $A_{/F}$.
Let $E_{\mathfrak{c}}/E$ denote the anticyclotomic extension of conductor $\mathfrak{c}$ defined in (\cite{plecticHeegner}, Section 4.2.1) with Galois group $\mathcal{G}_{\mathfrak{c}}=\Gal(E_{\mathfrak{c}}/ E)$.
Recall that for any character $\chi\colon \mathcal{G}_{\mathfrak{c}}\to \C^\times$ its \emph{conductor} is the maximal divisor $\mathfrak{c}_\chi$ of $\mathfrak{c}$ such that $\chi$ factors through $\mathcal{G}_{\mathfrak{c}}\onto \mathcal{G}_{\mathfrak{c}_\chi}$.
We write $\Q_{\chi}$ for the extension of $\Q$ generated by the values of $\chi$.
With a small change of notation compared to (\cite{plecticHeegner}, Section 4.4), we denote the \emph{plectic Stark--Heegner point} associated to an anticyclotomic character $\chi$ by
\[
\mP_{A,S}^{\chi}\in \widehat{A}(E_S)_{\Q_{\chi}}.
\] 
It follows easily from the construction of plectic Stark--Heegner points that there is an element
\[
\mP_{A,S}^{\mathcal{G}_\mathfrak{c}}\in \widehat{A}(E_S)\otimes_{\Z} \Q[\mathcal{G}_{\mathfrak{c}}]
\]
such that the equality $\chi(\mP_{A,S}^{\mathcal{G}_\mathfrak{c}})= \mP_{A,S}^{\chi}$
holds for any character $\chi\colon\mathcal{G}_{\mathfrak{c}}\to \C^\times$  of conductor $\mathfrak{c}$.
A similar statement also holds for plectic $p$-adic invariants (\cite{plecticHeegner}, Section 4.2): there is an element
\[
\mQ_{A,S}^{\mathcal{G}_\mathfrak{c}}\in \widehat{E}_{S,\otimes}^{-}\otimes_{\Z} \Z[\mathcal{G}_{\mathfrak{c}}]
\]
such that the equality $\chi(\mQ_{A,S}^{\mathcal{G}_\mathfrak{c}})= \mQ_{A,S}^{\chi}$
holds for any character $\chi\colon \mathcal{G}_{\mathfrak{c}}\to \C^\times$ of conductor $\mathfrak{c}$. Moreover, the two elements are related by the following equation 
\begin{equation}\label{pinvariantscomparison}
\phi^{\Tate}_S(\mQ_{A,S}^{\mathcal{G}_\mathfrak{c}})= \pr_S^{-}(\mP_{A,S}^{\mathcal{G}_\mathfrak{c}}).
\end{equation}

\subsection{Plectic points are Heegner points when $\lvert S\rvert=1$ }\label{Heegner}
The construction of plectic Stark--Heegner points generalizes the $p$-adic description of classical Heegner points given in (\cite{CDuniformization}, \cite{MokHeegner}). In this subsection, we recall the precise relation between the two constructions when the set $S$ consists of a single prime $\{\p\}$ and $\chi\colon \mathcal{G}_{\mathfrak{c}}\to \C^\times$ is a character of conductor $\frak{c}$.

\noindent  
We fix embeddings $\iota_\p\colon \overline{\Q}\into \overline{\Q}_p$ and $\iota_\nu\colon \overline{\Q}\into \C$ respectively inducing the $p$-adic prime $\p$ and an Archimedean place $\nu$ of $F$.
Since $\p$ is inert in $E/F$, it splits completely in the anticyclotomic extension $E_{\mathfrak{c}}/E$.
Thus, the embedding $\iota_\p$ restricts to an embedding $\iota_\wp\colon E_{\mathfrak{c}}\hookrightarrow E_{\p}$ inducing an injective homomorphism 
\[
\iota_{A,\p}\colon A(E_{\mathfrak{c}})_{\Q_\chi}\too \widehat{A}(E_{\p})_{\Q_\chi}.
\]
We write $r_\alg(A/E,\chi)$ for the $\Q_\chi$-dimension of the $\chi$-component $A(E_{\mathfrak{c}})^\chi_{\Q_\chi}$ and define $r_\an(A/E,\chi)$ to be the order of vanishing of the $L$-function $L(A/E,\chi,s)$ at $s=1$.
As explained in (\cite{AutomorphicDarmon}, Appendix A.1.2), there exists a Heegner point $P_{A,\mathfrak{c}}\in A(E_{\mathfrak{c}})$ of conductor $\mathfrak{c}$, arising from a Shimura curve associated to the $F$-quaternion algebra ramified exactly at $\p\n^{\flat}$ and all Archimedean places different from $\nu$, such that the image of
\begin{equation}\label{Heegnerpoint}
P_{A}^{\chi} := \sum_{\sigma \in \mathcal{G}_\mathfrak{c}} \chi^{-1}(\sigma)\cdot \sigma(P_{A,\mathfrak{c}})\ \in A(E_{\mathfrak{c}})^{\chi}_{\Q_\chi}
\end{equation}
 under $\iota_{A,\p}$ is a non-zero rational multiple of the plectic Stark--Heegner point associated to the triple $(A,E,\{\p\})$, i.e., there exists $k_{A,\p}^{\chi}\in\Q^\times$ such that:
\begin{equation}\label{comparison}
k_{A,\p}^{\chi}\cdot \mP_{A,\{\p\}}^{\chi}= \iota_{A,\p}\big(P_{A}^{\chi}\big).
\end{equation}

\begin{proposition}\label{Zhang}
Let $\chi\colon \mathcal{G}_\mathfrak{c}\to \C^\times$ be a character of conductor $\mathfrak{c}$.
We have
\[
\mP_{A,\{\p\}}^\chi\neq 0 \quad \iff\quad  r_{\an}(A/E,\chi)=1,
\]
and both statements imply $r_{\alg}(A/E,\chi)=1$. 
Moreover, if $\chi$ is not quadratic,
\[
r_{\an}(A/E,\chi)=1 \quad \implies \quad \mQ_{A,\{\p\}}^\chi\neq 0. 
\] 
\end{proposition}
\begin{proof}
By equation \eqref{comparison}, the equivalence between the non-triviality of $\mP_{A,\{\p\}}^\chi$ and the analytic rank one statement follows from (\cite{ZhangGZ}, Theorem 1.2.1), while the relation with the algebraic rank is a consequence of the main theorem in \cite{NekovarKolyvagin}. For the second claim, we begin by noting that the involution $A(E_{\mathfrak{c}})\to A(E_{\mathfrak{c}}),\ P \mapsto \overline{P},$ induced by the complex conjugation associated to the Archimedean place $\nu$, yields an isomorphism
\[
A(E_{\mathfrak{c}})^{\chi}_{\Q_\chi} \xlongrightarrow{\sim} A(E_{\mathfrak{c}})^{\chi^{-1}}_{\Q_\chi}.
\]
Then, we observe that equations \eqref{pinvariantscomparison} and \eqref{comparison} imply the equality
\begin{equation}\label{eqHeeg}
k_{A,\p}^{\chi}\cdot\phi^{\Tate}_\wp\big(\mQ_{A,\{\p\}}^\chi\big)
= \iota_{A,\p}\big(P_{A}^{\chi}-a_\wp\cdot\overline{P_{A}^{\chi}}\big)
\end{equation}
because  $\sigma_\wp\circ\iota_{A,\p}\big(P_{A}^{\chi}\big)=\iota_{A,\p}\big(\overline{P_{A}^{\chi}}\big)$ by (\cite{CDuniformization}, Theorem 4.7).
Now, our assumption is that $r_{\an}(A/E,\chi)=1$, and (\cite{ZhangGZ}, Theorem 1.2.1) implies that $P_{A}^\chi\neq 0$.
When $\chi$ is not quadratic, the intersection $A(E_{\mathfrak{c}})^\chi_{\Q_\chi}\cap A(E_{\mathfrak{c}})^{\chi^{-1}}_{\Q_\chi}$ is trivial and the claim follows.
\end{proof}

\begin{remark}\label{signs}
Let $\chi=\cf$ be the trivial character and assume that $r_\mathrm{an}(A/E)=1$. Using equation \eqref{eqHeeg} and (\cite{MokHeegner}, Corollary 4.2), we deduce that $\pr^{-}_{\wp}(\mP_{A,\{\p\}}^\cf)\not=0$ is equivalent to either
\[
\Big(a_\wp=+1\quad\&\quad r_\alg(A/F)=0\Big)\quad \text{or}\quad \Big(a_\wp=-1\quad\&\quad r_\alg(A/F)=1\Big). \]
\end{remark}

\subsection{The anticyclotomic Gross--Zagier formula}\label{antiGZ}
Let $E_{\mathfrak{c},S}$ be the union of the anticyclotomic extensions of $E$ of conductor $\mathfrak{c}\cdot \prod_{\p\in S}\p^n$ for $n\geq 0$.
We put $\mathcal{G}_{\mathfrak{c},S}=\Gal(E_{\mathfrak{c},S}/ E)$ and denote by
$$I_\mathfrak{c}\subseteq \Z_p\llbracket \mathcal{G}_{\mathfrak{c},S}\rrbracket$$
the relative augmentation ideal with respect to the quotient map $\mathcal{G}_{\mathfrak{c},S} \onto \mathcal{G}_{\mathfrak{c}}$.
Restriction of the global Artin homomorphism to the local components at $\p \in S$ induces the homomorphism
\[
\rec_S\colon \bigoplus_{\p\in S}E_\p^{-} \too \mathcal{G}_{\mathfrak{c},S}.
\]
Since $E/F$ is a quadratic CM extensions, the kernel of $\rec_S$ is torsion and its image is an open subgroup of $\ker(\mathcal{G}_{\mathfrak{c},S} \onto \mathcal{G}_{\mathfrak{c}})$.
Moreover, Lemma \ref{symmaps} \eqref{symmapsa} and Lemma \ref{injective} imply that the map
\[
\dd\rec_{S}\colon\widehat{E}_{S,\otimes}^{-}\otimes_{\Z_p}\Z_p[\mathcal{G}_{\mathfrak{c}}]
\too \Sym^{r}_{\Z_p}(\widehat{E}_{\p_1}^{-}\otimes_{\Z_p}\ldots \otimes_{\Z_p} \widehat{E}_{\p_r}^{-})\otimes_{\Z_p}\Z_p[\mathcal{G}_{\mathfrak{c}}]
\too \left(I_{\mathfrak{c}}^{r}/I_{\mathfrak{c}}^{r+1}\right)_\Q
\]
is injective.
Associated to the quadruple $(A,S,E,\mathfrak{c})$ there is the square-root $p$-adic $L$-functions 
\[
\mathscr{L}_S(A/E)_\mathfrak{c} \in \Z\llbracket\mathcal{G}_{\mathfrak{c},S}\rrbracket
\]
constructed in (\cite{BG2}, Definition 5.2) and also in (\cite{plecticHeegner}, Definition 5.7).
On the one hand, the values of this $p$-adic $L$-function at finite order characters not satisfying certain ramification conditions are always equal to zero (see \cite{plecticHeegner}, Theorem 5.10), while the (squares of the) non-trivial values are explicitly calculated in (\cite{BG2}, Theorem 5.8). We recall the interpolation formula in Theorem \ref{thm:specialvalues} below. For the rest of this subsection we consider $\mathscr{L}_S(A/E)_\mathfrak{c}$ as an element of $\Z_p\llbracket\mathcal{G}_{\mathfrak{c},S}\rrbracket$.
Recall that (\cite{BG2}, Theorem 5.5) shows that
\[
\mathscr{L}_S(A/E)_\mathfrak{c}\in I_{\mathfrak{c}}^{r}.
\]
The following is a reformulation of the main theorem of \cite{plecticHeegner}.
\begin{theorem}[(Anticyclotomic Gross-Zagier formula)]\label{GZ}
The equality
\[
2^{r}\cdot \partial_{\mathcal{G}_\mathfrak{c}}^{r}(\mathscr{L}_{S}(A/E)_\mathfrak{c}) = \dd\rec_{S}\big((\mQ_{A,S}^{\mathcal{G}_\mathfrak{c}})^\vee\big)
\]
holds in $(I_\mathfrak{c}^{r}/I_\mathfrak{c}^{r+1})_\Q$.
\end{theorem}

\noindent Interestingly, (\cite{BG2}, Proposition 5.6) allows us to describe the behaviour of $\mathscr{L}_{S}(A/E)_\mathfrak{c}$ under the involution $(-)^\vee$ in terms of the global root number $\varepsilon(A/F)$ of $A_{/F}$ and a product of local root numbers $\varepsilon_S(A/F):=\prod_{\p\in S}\varepsilon_\p(A/F)$.
\begin{proposition}\label{involution}
The equality
\[(\mathscr{L}_{S}(A/E)_\mathfrak{c})^{\vee}= \varepsilon(A/F) \cdot \varepsilon_S(A/F) \cdot \mathscr{L}_{S}(A/E)_\mathfrak{c}\]
holds up to multiplication with an element in $\mathcal{G}_{\mathfrak{c},S}$. 
\end{proposition}

\begin{corollary}\label{involutiononpoints}
There exists an element $g\in \mathcal{G}_{\mathfrak{c}}$ such that the equality
\[
\mQ_{A,S}^{\chi^{-1}}= \chi(g)\cdot \varepsilon(A/F) \cdot \varepsilon_S(A/F) \cdot (-1)^{r}\cdot \mQ_{A,S}^{\chi}
\]
holds.
In particular, if $\chi=\cf$ is the trivial character, we have
\[
\mQ_{A,S}^{\cf}\neq 0\quad \implies\quad (-1)^r = \varepsilon(A/F) \cdot \varepsilon_S(A/F).
\]
\end{corollary}
\begin{proof}
Thanks to the commutative diagram \eqref{diagramSection2.2} and Theorem \ref{GZ} we have
\begin{align*}
\chi^{-1}(\dd\rec_{S}(\mQ_{A,S}^{\mathcal{G}_\mathfrak{c}}))
&=(-1)^{r}\chi(\dd\rec_{S}(\mQ_{A,S}^{\mathcal{G}_\mathfrak{c}})^{\vee})\\
&=(-2)^{r}\chi(\partial_{\mathcal{G}_\mathfrak{c}}^{r}(\mathscr{L}_{S}(A/E)_\mathfrak{c})).
\end{align*}
Proposition \ref{involution} implies that there is $g \in \mathcal{G}_{\mathfrak{c}}$ such that
\[
\chi(\partial_{\mathcal{G}_\mathfrak{c}}^{r}(\mathscr{L}_{S}(A/E)_\mathfrak{c}))
=\chi(g)\cdot \varepsilon(A/F) \cdot \varepsilon_S(A/F)\cdot  \chi(\partial_{\mathcal{G}_\mathfrak{c}}^{r}\mathscr{L}_{S}(A/E)_\mathfrak{c}^\vee),
\]
and applying Theorem \ref{GZ} once more we obtain the equality
\[
\chi^{-1}(\dd\rec_{S}(\mQ_{A,S}^{\mathcal{G}_\mathfrak{c}})) =\chi(g)\cdot \varepsilon(A/F) \cdot \varepsilon_S(A/F) \cdot (-1)^{r}\cdot \chi(\dd\rec_{S}(\mQ_{A,S}^{\mathfrak{c}})).
\]
The claim now follows from the injectivity of $\dd\rec_{S}$.
\end{proof}

\begin{remark}
Corollary \ref{involutiononpoints} proves one implication of  (\cite{PlecticInvariants}, Conjecture 1.6) in the CM case.
\end{remark}

\subsection{Interpolation formula}
The square-root $p$-adic $L$-function $\mathscr{L}_S(A/E)_\mathfrak{c}$ does not interpolate values of complex $L$-functions but -- as the name suggests -- a choice of their square-roots.
\begin{definition}
The anticyclotomic $p$-adic $L$-function is defined as the product
	\[
	\mathfrak{L}_S(A/E)_\mathfrak{c}:=\mathscr{L}_S(A/E)_\mathfrak{c}\cdot\mathscr{L}_S(A/E)_\mathfrak{c}^\vee\ \in  \Z\llbracket\mathcal{G}_{\mathfrak{c},S}\rrbracket.
	\]
\end{definition}
\noindent It is this $p$-adic $L$-function that interpolates special values of complex $L$-functions.
We introduce some notation to state the interpolation property:
the normalized special value of the complex $L$-function $L(A/E,\chi,s)$ at $s=1$ is given by
\[
L^{\nm}(A/E,\chi):= \frac{L
(A/E, \chi,1)}{L(\pi_A,\Ad,1)}\cdot 
\prod_{\nu\mid\infty} C_\nu(E, \pi_{A},\chi) 
\]
where $L(\pi_A, \Ad, s)$ is the adjoint $L$-function of the $\GL_{2,F}$-automorphic representation $\pi_A$ attached to $A_{/F}$, and $C_\nu(E, \pi_{A},\chi)$ is the Archimedean factor defined in (\cite{FileMartinPitale}, Section 7B).
Since the quadratic extension $E/F$ is CM, the factors $C_\nu(E, \pi_{A},\chi)$ are all equal to a fixed constant $C_\infty$ independent of $\nu$, $A_{/F}$, $E$ and $\chi$.
Note that we always include the Archimedean Euler factors in the definition of the $L$-functions and $\zeta$-functions that occur.

\noindent By our assumptions, the automorphic representation $\pi_A$ admits a Jacquet--Langlands lift to group of units of the totally definite quaternion algebra $B_{/F}$ that is ramified exactly at those finite primes that divide $\n^{\flat}$.
As $B_{/F}$ is totally definite, we may normalize the newform of the Jacquet-Langlands lift such that it takes rational values.
Let $f_A^{B}$ be the rational newform involved in the construction of the $p$-adic $L$-function in \cite{BG2}.
\begin{theorem}\label{thm:specialvalues}
For all locally constant characters $\chi\colon \mathcal{G}_{\mathfrak{c},S} \to \C^\times$ we have
\begin{align*}
\chi(\mathfrak{L}_S(A/E)_\mathfrak{c})=
&\frac{1}{2} \cdot   \langle f_A^{B}, f_A^{B} \rangle_B \cdot   L_{\Sigma_A}(1,\omega_{E/F}) \cdot  L_{\Sigma_A^{\add}}(\pi_A,\Ad,1)\cdot\prod_{\q \in \Sigma_A\setminus S} \hspace{-0.5em}e_{\q}(E/F) \\
&\times \zeta_F^{\Sigma_A}(2)\cdot \sqrt{\frac{\Delta_F}{\Delta_E}}\cdot  L^{\nm}(A/E, \chi,1)\cdot \prod_{\q \mid p_S\mathfrak{c}} C_{\ord_{\p}(\mathfrak{c}),\chi_\p, A_{\p}} 
\end{align*}
where
\begin{itemize}
\item $\langle\cdot,\cdot\rangle_B$ denotes the Petersson inner product on automorphic forms of the unit group of $B_{/F}$,
\item $\Sigma_A$ denotes the set of primes of $F$ at which $A_{/F}$ has bad reduction, and $\Sigma_A^{\add}\subset\Sigma_A$ the subset of primes of additive reduction,
\item $\omega_{E/F}$ is the quadratic character associated to $E/F$,
\item $e_{\q}(E/F)$ denotes the ramification degree of $\q$ in $E/F$,
\item $\Delta_F$ are $\Delta_E$ the absolute values of the discriminants of $F$ and $E$ respectively,
\item $C_{\ord_{\p}(\mathfrak{c}),\chi_\p, A_{F_\p}}\in \Q_\chi$ are constants that only depend on the the $\p$-adic valuation of $\mathfrak{c}$, the restriction of $\chi_\p$ to a decomposition group at $\p$, and the base change $A_{/F_\p}$.
\end{itemize}
\end{theorem}
\begin{proof}
In case the conductor of $\chi$ is exactly $\mathfrak{c}$ this follows from (\cite{BG2}, Theorem 5.8), and the explicit Waldspurger formula of \cite{FileMartinPitale}.
The general case follows from the norm relations of (\cite{BG2}, Theorem 5.8).
\end{proof}

%%%%%%%%%%%%%%%%%%%%%%%%%%%%%%%%%%%%%%%%%%%%%%%%%%%%%%
% Artin formalism
%%%%%%%%%%%%%%%%%%%%%%%%%%%%%%%%%%%%%%%%%%%%%%%%%%%%%%
\section{Artin formalism for plectic points}\label{ArtinPlectic}
We keep the same notation as in Subsection \ref{antiGZ}.
In addition, we assume that we are in the polyquadratic setup of Subsection \ref{polyquadratic}, and that $\mathfrak{c}=c \cdot \mathcal{O}_F$ for some ideal $c$ of $F_\circ$.
Let $E_{\circ,c}$ denote the ring class field of conductor $c$ and $E_{\circ,c,\{\wp\}}$ the union of the ring class fields of $E_\circ$ of conductor $c\wp^n$ for $n\geq 0$.
We put $G_{c}=\Gal(E_{\circ,c}/ E_\circ)$ and  $G_{c,\wp}=\Gal(E_{\circ,c,\{\wp\}}/ E_\circ)$, so that there are natural maps $\mathcal{G}_\mathfrak{c}\to G_c$ and $\mathcal{G}_{\mathfrak{c},S}\to G_{c,\wp}.$
\begin{assumptions}\label{AssRamification}
	We require that
	\begin{itemize}
		\item every prime divisor of $c$ is completely split in $F/F_\circ$ and unramified in $E/F_\circ$,
		\item the character $\chi\colon\mathcal{G}_\mathfrak{c}\to\C^{\times}$ is the restriction of an anticyclotomic character $\xi\colon G_c\to\C^{\times}$.
	\end{itemize}
\end{assumptions} 
\begin{remark}\label{ExplanationAss}
Since every non-trivial subextension of $F/F_\circ$ is ramified by Assumption \ref{AssFields}, requiring the ideal $c$ to be unramified in $E/F_\circ$ implies that the fields $E$ and $E_{\circ,c}$ are linearly disjoint over $E_\circ$. We impose the total splitting in $F/F_\circ$ of the prime divisors of $c$ just as a simplifying hypothesis for the proof of Proposition \ref{factorizationpLf}.
\end{remark}

\subsection{Factorization of complex $L$-functions}
We write $\mathfrak{G}^\star:=\Hom_\mathrm{gr}(\mathfrak{G},\{\pm 1\})$ for the Pontryagin dual of $\mathfrak{G}=\Gal(F/F_\circ)\cong\Gal(E/E_\circ)$.
For any number field $\Omega$ we write $G_\Omega:=\Gal(\overline{\Q}/\Omega)$ for its absolute Galois group.
\begin{lemma}\label{funrepth}
	There exists an isomorphism of $G_{F_\circ}$-representations:
	\[
	\Ind_{G_E}^{G_{F_\circ}}(\chi)\cong \bigoplus_{\eta\in \mathfrak{G}^\star}\Ind_{G_{E_\circ}}^{G_{F_\circ}}(\xi\cdot\eta).
	\]
\end{lemma}
\begin{proof}
	First, we claim that 
		\[
		\Ind_{G_E}^{G_{E_\circ}}(\chi)\cong \bigoplus_{\eta\in \mathfrak{G}^\star} \xi \cdot \eta.
		\]
		To prove the claim note that all the characters $\xi \cdot \eta$, for $\eta\in \mathfrak{G}^\star$,
		have the same restriction to $G_E$, namely the character $\chi$.
		By Frobenius reciprocity we have
		\[
		\Hom_{G_{E_\circ}}\big(\xi \cdot \eta, \Ind_{G_E}^{G_{E_\circ}}(\chi)\big) \cong \Hom_{G_E}\big(\chi,\chi\big)\not=0
		\]
		and, therefore, the claim follows from semi-simplicity of representations of finite groups in characteristic zero. The statement of the lemma follows from transitivity of induction.
\end{proof}

\begin{corollary}\label{factorpadicL}
	The following equality of complex $L$-functions holds:
	\[
	L^{\nm}(A/E,\chi,s)=\prod_{\eta\in\mathfrak{G}^\star}L^{\nm}(A^{\eta}_{\circ}/E_\circ,\xi,s).
	\]
\end{corollary}
\begin{proof}
This is a direct consequence of Lemma \ref{funrepth} and Artin formalism for complex $L$-functions.
\end{proof}

\subsection{Factorization of $p$-adic $L$-functions}
For each $\eta\in\mathfrak{G}^\star$, the quadruple $(A_\circ^{\eta},E_\circ,\wp,c)$ fulfils the conditions of Subsection \ref{antiGZ} and thus, we can define the $p$-adic L-functions
\[
\mathscr{L}_{\{\wp\}}(A_\circ^{\eta}/E_\circ)_c,\ \mathfrak{L}_{\{\wp\}}(A_\circ^{\eta}/E_\circ)_c \in \Z\llbracket G_{c,\wp}\rrbracket.
\]
Let $R$ be a commutative ring.
The homomorphism $\mathcal{G}_{\mathfrak{c},S}\too G_{c,\wp}$ induces the restriction map
\[
\res_{S,\wp}\colon R\llbracket \mathcal{G}_{\mathfrak{c},S} \rrbracket \too R\llbracket G_{c,\wp}\rrbracket
\]
between completed group algebras.

\begin{proposition}\label{factorizationpLf}
There is a constant $C\in\Q^\times$ such that the equality
	\[
	\res_{S,\wp}\left(\mathfrak{L}_S(A/E)_\mathfrak{c}\right)=C\cdot \prod_{\eta\in\mathfrak{G}^\star}\mathfrak{L}_{\{\wp\}}(A_\circ^{\eta}/E_\circ)_c.
	\]
holds in $\Z\llbracket G_{c,\wp}\rrbracket$.
\end{proposition}
\begin{proof}
To prove the statement it is enough to use Theorem \ref{thm:specialvalues} to show equality of both sides after evaluation at every finite order character of $G_{c,\wp}$.
We note that the first line of the interpolation formula is a non-zero rational number and, thus, we may neglect it.
By Corollary \ref{factorpadicL}, the normalized special values on both sides cancel out.
In addition, the local constants $C_{\ord_{\p}(\mathfrak{c}),\chi_\p, A_{\p}}$ cancel out as well because we assumed that all the primes of $F_\circ$ dividing $c\wp$ are totally split in $F/F_\circ$.
As we are taking a product over $r=2^t$ factors on the right hand side, we see that it suffices to show that
\[
\sqrt{\frac{\Delta_F}{\Delta_E}}\cdot \frac{\zeta_F(2)}{\zeta_{F_\circ}(2)^{r}}
\]
is a rational number.
Using the functional equation of the Dedekind zeta function, we may rewrite this term as
\[
\sqrt{\frac{\Delta_F}{\Delta_E}}\cdot \frac{\zeta_F(-1)\Delta_{F}^{-3/2}}{\zeta_{F_\circ}(-1)^r\Delta_{F_\circ}^{-3r/2}}.
\]
By the Klingen--Siegel Theorem (\cite{Siegel}, \cite{Klingen}), the special values of the Dedekind zeta function (excluding the Archimedean Euler factors) at negative integers are rational.
The Archimedean factors of $\zeta_F(s)$ and $\zeta_{F_\circ}(s)^{r}$ agree, so we are left to show that $\Delta_E\in \Q^\times$ is a square.
Let $\Delta_{E/F_\circ}$ denote the relative discriminant of $E/F_\circ$. The formula for relative discriminants in towers gives
\[
\Delta_E= \No_{F_\circ/\Q}(\Delta_{E/F_\circ}) \cdot \Delta_{F_\circ}^{[E :F_\circ]}.
\]
As $[E:F_\circ]$ is a power of $2$, it suffices to show that $\Delta_{E/F_\circ}$ is a square.
At this point we conclude the argument because (\cite{discriminants}, Theorem 1.2) implies
\[
\Delta_{E/F_\circ}=\Delta_{E_\circ/F_\circ}^{r} \cdot \Delta_{F/F_\circ}^2
\]
since the relative discriminants $\Delta_{E_\circ/F_\circ}$ and $\Delta_{F/F_\circ}$ are coprime by Assumption \ref{AssFields}.
\end{proof}

\subsection{Factorization of plectic invariants}
Recall that by Assumption \ref{AssRamification} the character $\chi$ is the restriction to $\mathcal{G}_\mathfrak{c}$ of an anticyclotomic character $\xi\colon G_c\to\C^\times$.
For every $\eta\in\mathfrak{G}^\star$ we can consider the \emph{plectic $p$-adic invariant} 
\[
\mQ_{A_\circ^{\eta},\{\wp\}}^{\xi} \in   \widehat{E}_{\circ,\wp}^{-}\otimes_{\Z}\Q_{\xi}
\]
attached to the triple $(A_\circ^{\eta}/F_\circ, \xi, \{\wp\})$. 
As in Subsection \ref{mainresults}, we use the fixed identifications
$E_\p^{-}=E_{\circ,\wp}^{-}$ for $\p\in S$
to define the \emph{norm map}
\[
\No_{S/\wp}\colon \widehat{E}_{S,\otimes}^{-}\too (E_{\circ,\wp}^{-})^{\otimes r}
\too \Sym^{r}_{\Z_p} (\widehat{E}_{\circ,\wp}^{-}).
\]

\begin{corollary}\label{keyidentitysquare}
There exists a constant $C_\chi\in\Q_\chi^\times$ such that the equality
	\[
	\No_{S/\wp}\big(\mQ_{A,S}^\chi\big)^2=C_\chi \cdot\prod_{\eta\in\mathfrak{G}^\star} \big(\mQ_{A^{\eta}_\circ,\{\wp\}}^\xi\big)^2
	\]
	 holds in $\Sym^{2r}_{\Z_p}(\widehat{E}_{\circ,\wp}^{-})_{\Q_{\chi}}$. 
	Moreover, $\No_{S/\wp}(\mQ_{A,S}^\chi)\neq 0$ implies $\mQ_{A^{\eta}_\circ,\{\wp\}}^\xi\neq 0$ for all $\eta\in\mathfrak{G}^\star$.
\end{corollary}
\begin{proof}
	The second claim follows from the first because $\Sym_{\Z_p}(\widehat{E}_{\circ,\wp}^{-})_{\Q_\chi}$ has no nilpotent elements.
	By Proposition \ref{factorizationpLf} and Proposition \ref{involution} there exists $C\in\Q^\times$ such that the equality
	\[\res_{S,\wp}\left(\mathscr{L}_S(A/E)_\mathfrak{c}\right)^2=C\cdot \prod_{\eta\in\mathfrak{G}^\star}\mathscr{L}_{\{\wp\}}(A_\circ^{\eta}/E_\circ)_c^2\]
	holds up to multiplication by an element in $G_{c,\wp}$.
	Lemma \ref{symmaps} \eqref{symmapsb} implies that the diagram
	\[\xymatrix{
	\widehat{E}_{S,\otimes}^{-}\ar[rr]^{\dd\rec_{S}}\ar[d]_{\No_{S/\wp}} && (I_\mathfrak{c}^{r}/I_\mathfrak{c}^{r+1})_\Q\ar[d]^{\res_{S,\wp}}\\
	\widehat{E}_{\circ,\wp}^{-}\ar[rr]^{\dd\rec_{\wp}} && (I_c^{r}/I_c^{r+1})_\Q
	}\]
	commutes. Thus, we may conclude by invoking Theorem \ref{GZ}.
\end{proof}

\noindent Assume that $p$ does not split in the field $\Q_\chi$, i.e., that $\Q_{\chi,p}:= \Q_{\chi}\otimes_\Q\Q_p$ is a field.
Then, Lemma \ref{squaremap} allows us to deduce the following corollary.
\begin{corollary}\label{keyidentity}
Assume that $p$ does not split in $\Q_\chi$, then
there exists a square-root $\sqrt{C_\chi}\in \Q_{\chi,p}^\times$ of $C_\chi\in \Q_\chi$ such that the equality
\[
\No_{S/\wp}(\mQ_{A,S}^\chi)=\sqrt{C_\chi} \cdot\prod_{\eta\in\mathfrak{G}^\star} \mQ_{A^{\eta}_\circ,\{\wp\}}^\xi
\] 
holds in $\Sym^{r}_{\Z_p}(\widehat{E}_{\circ,\wp}^{-})_{\Q_{\chi}}$. Moreover, $\No_{S/\wp}(\mQ_{A,S}^\chi)\neq 0$ if and only if $\mQ_{A^{\eta}_\circ,\{\wp\}}^\xi\neq 0$ for all $\eta\in\mathfrak{G}^\star$.
\end{corollary}

%%%%%%%%%%%%%%%%%%%%%%%%%%%%%%%%%%%%%%%%%%%%%%%%%%%%%%
% Algebraicity of plectic points
%%%%%%%%%%%%%%%%%%%%%%%%%%%%%%%%%%%%%%%%%%%%%%%%%%%%%%
\section{Algebraicity and arithmetic significance of plectic points}
We keep the same notation as in Section \ref{ArtinPlectic}.
In particular, we are in the polyquadratic CM setting and Assumptions \ref{AssHeegner2}, \ref{AssFields}, \ref{AssHeegner} and \ref{AssRamification} hold.
Moreover, we suppose that the conductor of the character $\xi$ is $c$. One can always achieve this by shrinking $c$.
We have the equalities
	\begin{equation}\label{star}
	r_{\bfcdot}(A/E,\chi)=\sum_{\eta\in\mathfrak{G}^\star}r_{\bfcdot}(A_\circ^\eta/E_\circ,\xi),\qquad\bfcdot\in\{\an,\alg\},
	\end{equation}
which imply $r_{\an}(A/E,\chi)\geq r$ because $r_{\an}(A_\circ^{\eta}/E_\circ,\xi)$ is odd for every character $\eta\in\mathfrak{G}^\star$.

\begin{theorem}\label{GZKabove}
 The following implication holds:
 \[
 \No_{S/\wp}\big(\pr^{-}_S(\mP^\chi_{A,S})\big)\neq 0\quad\implies\quad r_{\alg}(A/E,\chi)=r\quad\&\quad r_{\an}(A/E,\chi)=r.
 \]
Suppose $p$ does not split in $\Q_\chi$ and that the character $\chi$ is not quadratic, then
 \[
 r_{\an}(A/E,\chi)=r \quad\implies\quad \No_{S/\wp}\big(\pr^{-}_S(\mP^\chi_{A,S})\big)\neq 0\quad\&\quad r_{\alg}(A/E,\chi)=r.
 \]
\end{theorem}
\begin{proof}
The non-vanishing of $\No_{S/\wp}\big(\pr^{-}_S(\mP^\chi_{A,S})\big)$ implies that for every character $\eta\in\mathfrak{G}^\star$ the point $\pr^{-}_S(\mP^\xi_{A_\circ^{\eta},\{\wp\}})$ is non-zero by Corollary \ref{keyidentitysquare}. Hence, the point $\mP^\xi_{A^{\eta},\{\wp\}}$  is non-zero for every $\eta\in\mathfrak{G}^\star$ and, thus, 
Proposition \ref{Zhang} gives
\[
r_\alg(A_\circ^\eta/E_\circ,\xi)=r_\an(A_\circ^\eta/E_\circ,\xi)=1 \quad\forall\ \eta\in\mathfrak{G}^\star.
\]
The first claim follows from \eqref{star}.
The second claim is proved similarly using Corollary \ref{keyidentity}.
\end{proof}

\noindent Set $A^+_{/F}=A_{/F}$ and denote by $A^-_{/F}$ the quadratic twist of $A_{/F}$ with respect to the extension $E/F$.
We partition $S=S^+\cup S^-$ by declaring that the subset $S^+\subseteq S$ contains \emph{all} the primes in $S$ of split multiplicative reduction for $A_{/F}^+$, and define $\varrho_{A}(S):=\max\big\{r_{\alg}(A^\pm/F)+\lvert S^\pm\rvert\big\}.$
 
 \begin{theorem}
  We also have
 \[ 
 \No_{S/\wp}\big(\pr^{-}(\mP^\cf_{A,S})\big)\neq0\quad\iff\quad r_{\an}(A/E)=r\quad\&\quad\varrho_A(S)=r.
 \]
 \end{theorem}
 \begin{proof}
Under our hypotheses we either have $S=S^+$ or $S=S^-$.
For simplicity we assume that $S=S^+$, or equivalently that $a_\wp=+1$, since the other case is proved by similar arguments.

\noindent Suppose $\No_{S/\wp}\big(\pr^{-}_S(\mP^\cf_{A,S})\big)$ is non-zero.
By Theorem \ref{GZKabove}, it is enough to show that $\varrho_S(A)=r$.
Now, Corollary \ref{keyidentity} implies that $\pr^{-}_{\wp}(\mP^\cf_{A_\circ^\eta,\wp})\neq 0$ for every $\eta\in\mathfrak{G}^\star$, and since $a_\wp=+1$,  Remark \ref{signs} tells us that $r_\alg(A_\circ^{\eta}/F_\circ)=0$ for every $\eta\in\mathfrak{G}^\star$.
We deduce $r_{\alg}(A^+/F)=0$ and $r_{\alg}(A^-/F)=r$ as required. For the converse implication, note that the assumptions imply that $r_{\alg}(A_\circ^\eta/E_\circ)=1$ and $r_{\alg}(A_\circ^\eta/F_\circ)=0$ for every $\eta\in\mathfrak{G}^\star$.
Then, Remark \ref{signs} shows that $\pr^{-}_{\wp}(\mP^\cf_{A_\circ^\eta,\wp})\neq 0$ for every $\eta\in\mathfrak{G}^\star$ and the claim follows from Corollary \ref{keyidentity}.
 \end{proof}

\noindent As in Subsection \ref{Heegner} we fix an embedding $\iota_\wp\colon E_{\circ,c}\into E_{\circ,\wp}$ extending the canonical embedding $E_\circ\into E_{\circ,\wp}$.
Furthermore, we choose an embedding $\iota_{\p_1}\colon E_{\mathfrak{c}} \into E_{\p_1}$ that restricts to the chosen embedding $\iota_\wp$ as well as to the canonical embedding $E\into E_{\p_1}$.
Under our assumptions, the fields $E$ and $E_{\circ,c}$ are linearly disjoint over $E_\circ$, thus for every $\p\in S$ we can choose an element $\tau_{\p}\in \Gal(E_\mathfrak{c}/E_{\circ,c})$ whose restriction to $\Gal(E/E_\circ)\cong\mathfrak{G}$ sends $\p$ to $\p_1$, and such that $\iota_{\p}:=\iota_{\p_1}\circ\tau_{\p}$ corresponds to the prime $\p$. 
Then, we define the determinant map $\det\colon\exterior{r}(A(E_\mathfrak{c})_\Q)\to \widehat{A}(E_{S})_\Q$ by setting
\[
\det\big(P_1\wedge\dots\wedge P_r\big)=\det \begin{pmatrix}
		\iota_{A,\mathfrak{p}_1}(P_1)&\dots& \iota_{A,\mathfrak{p}_r}(P_1)\\
	\vdots&\ddots&\vdots\\
	\iota_{A,\mathfrak{p}_{1}}(P_r)&\dots& \iota_{A,\mathfrak{p}_r}(P_r)
	\end{pmatrix}.
\]
For every $\eta\in\mathfrak{G}^\star$ we may view $A_\circ^\eta(E_{\circ,c})$ as a subgroup of $A(E_\mathfrak{c})$, and by construction we have
\begin{align}\label{Galoisactions}
\iota_{A,\p}(P)=\eta(\tau_{\p})\cdot \iota_{A_\circ^\eta,\wp}(P)\qquad \forall\ P\in A_\circ^\eta(E_{\circ,c}).
\end{align}

\begin{theorem}\label{Algebraicity}
 Suppose $p$ does not split in $\Q_\chi$.
 There exists a quadratic extension $\Omega_\chi/\Q_\chi$, in which $p$ splits, and an element $w^\chi_{A,S}\in \exterior{r} (A(E_\mathfrak{c})^\chi_{\Omega_\chi})$ such that
	\[
	\No_{S/\wp}\big(\pr^{-}_S(\mP^\chi_{A,S})\big)=\No_{S/\wp}\big(\pr^{-}_S(\det(w_{A,S}^\chi))\big)
	\]
	 in $\Sym^r_{\Z_p}(\widehat{A}(E_{\circ,\wp}))_{\Q_\chi}.$
\end{theorem}
\begin{proof}
We choose an ordering $\eta\{1\},\ldots,\eta\{r\}$ of the elements of the character group $\mathfrak{G}^\star$.
For $\eta\in\mathfrak{G}^\star$ we consider the Heegner point $P_{A^\eta_\circ}^\xi$ of equation \eqref{Heegnerpoint}, and set
\[
\widetilde{w}^\chi_{A,S}:=P_{A^{\eta\{1\}}_\circ}^\xi \wedge \ldots \wedge P_{A^{\eta\{r\}}_\circ}^\xi.
\]
Using \eqref{Galoisactions}, the formula for the determinant gives 
\begin{align*}
\No_{S/\wp}\big(\det(\widetilde{w}^\chi_{A,S})\big)
&= \No_{S/\wp}\left(\sum_{\sigma \in S_r} \mathrm{sgn}(\sigma)\cdot  \iota_{\p_1}\big(P_{A^{\eta\{\sigma(1)\}}_\circ}^\xi\big)\otimes\ldots\otimes \iota_{\p_r}\big(P_{A^{\eta\{\sigma(r)\}}_\circ}^\xi \big)\right)\\
&=C_{\mathfrak{G}} \cdot \prod_{\eta\in \mathfrak{G}^\star}  \iota_{A_\circ^\eta,\wp}\big(P_{A^\eta_\circ}^\xi\big),
\end{align*}
where
\[
C_{\mathfrak{G}}:=\sum_{\sigma \in S_r} \mathrm{sgn}(\sigma) \cdot \prod_{i=1}^{r} \eta\{\sigma(i)\}(\tau_{\p_{i}})
\]
is the determinant of the character table of the group $\mathfrak{G}$.
By orthogonality of characters, the determinant is non-zero.
In fact, it is equal to $\pm r^{r/2}$.
Let $k_{A_\circ^\eta,\wp}^\xi\in \Q^\times$ be the constants appearing in \eqref{comparison}, and $\sqrt{C_\chi}\in\Q_{\chi,p}^\times$ the constant appearing in Corollary \ref{keyidentity}.
By setting
	\[
	w^\chi_{A,S}:=\frac{\sqrt{C_\chi} }{C_{\mathfrak{G}} \cdot \prod_{\eta\in\mathfrak{G}^\star} k_{A_\circ^\eta,\wp}^\xi}\cdot \widetilde{w}^\chi_{A,S}
	\]
we get that
\[
\No_{S/\wp}\big(\det(w^\chi_{A,S})\big)= \sqrt{C_\chi} \cdot \prod_{\eta\in \mathfrak{G}^\star}  \mP_{A^\eta_\circ,\{\wp\}}^\xi .
\]
The claim follows from Corollary \ref{keyidentity} after applying the minus projector on both sides.
\end{proof}

%%%%%%%%%%%%%%%%%%%%%%%%%%%%%%%%%%%%%%%%%%%%%%%%%%%%%%
% BIBLIOGRAPHY
%%%%%%%%%%%%%%%%%%%%%%%%%%%%%%%%%%%%%%%%%%%%%%%%%%%%%%

\bibliography{Plectic}
\bibliographystyle{alpha}

\end{document}